\documentclass[11pt]{amsart}
\usepackage{palatino}
\usepackage[v2]{xy}

\newtheorem{theorem}{Theorem}[section]
\newtheorem{corollary}[theorem]{Corollary}
\newtheorem{proposition}[theorem]{Proposition}
\theoremstyle{definition}
\newtheorem{definition}[theorem]{Definition}
\newtheorem{remark}[theorem]{Remark}
\newtheorem{example}[theorem]{Example}

\usepackage{mathrsfs}

\newcommand{\Cc}{{\mathscr C}}
\newcommand{\Ee}{{\mathscr E}}
\newcommand{\Bb}{{\mathscr B}}
\newcommand{\Mm}{{\mathscr M}}
\newcommand{\Aa}{{\mathscr A}}

\usepackage{bbold}

\newcommand{\en}{{\mathbb n}}
\newcommand{\zero}{{\mathbb 0}}
\newcommand{\one}{{\mathbb 1}}
\newcommand{\two}{{\mathbb 2}}

\newcommand{\bbb}{{\mathbb B}}
\newcommand{\eee}{{\mathbb E}}
\newcommand{\nn}{{\mathbb N}}
\newcommand{\Z}{{\mathbb Z}}
\newcommand{\aaa}{{\mathbb T}}

\newcommand{\Hom}{{\rm Hom}}
\newcommand{\llim}{{\rm lim}}
\newcommand{\clim}{{\rm colim}}
\newcommand{\Ext}{{\rm Ext}}
\newcommand{\Tor}{{\rm Tor}}
\newcommand{\Mor}{{\rm Mor}}

\newcommand{\h}{{\mathcal H}}

\newcommand{\Cat}{{\mathfrak{Cat}}}
\newcommand{\Fun}{{\mathfrak{Fun}}}
\newcommand{\PsdFun}{{\mathfrak{PsdFun}}}
\newcommand{\Fib}{{\mathfrak{Fib}}}
\newcommand{\Ab}{{\mathfrak{Ab}}}
\newcommand{\Nat}{{\mathfrak{Nat}}}
\newcommand{\theories}{{\mathfrak{Theories}}}

\begin{document}
\def\ead{\email}

\title[Andr\'e spectral sequences for Baues-Wirsching cohomology]{Andr\'e spectral sequences for Baues-Wirsching cohomology of categories}

\author{Imma G\'alvez-Carrillo}
\address{Departament de Matem\`atica Aplicada III,
      Universitat Polit\`ecnica de Cata\-lunya\\ Escola d'Enginyeria de Terrassa, Carrer Colom 1, 08222 Terrassa (Bar\-celona), Spain}
\ead{m.immaculada.galvez@upc.edu}

\author{Frank Neumann}
\address{Department of Mathematics, University of Leicester\\
University Road, Leicester LE1 7RH, England, United Kingdom}
\ead{fn8@mcs.le.ac.uk}

\author{Andrew Tonks
}
\address{Faculty of Computing, London Metropolitan University\\
166--220 Holloway Road, London N7 8DB, United Kingdom}
\ead{a.tonks@londonmet.ac.uk}

\begin{abstract}
  We construct spectral sequences in the framework of Baues-Wirsching
  cohomology and homology for functors between small categories and
  analyze particular cases including Grothendieck fibrations. We also
  give applications to more classical cohomology and homology theories
  including Hochschild-Mitchell cohomology and those studied before by
  Watts, Roos, Quillen and others.
\end{abstract}


\maketitle

\section*{Introduction}
In a fundamental paper  \cite{BW}, Baues and Wirsching introduced a very general version of cohomology for small categories. This Baues-Wirsching cohomology uses general coefficients given by natural systems. A natural system on a small category $\Cc$ is a functor from the factorization category $F(\Cc)$ into the category $\Ab$ of abelian groups. Baues-Wirsching cohomology generalizes at once many of the previously constructed versions of cohomology for categories, including Hochschild-Mitchell cohomology \cite{Mi}, the various versions of cohomology of small categories introduced by Watts \cite{Wa}, Grothendieck \cite{G}, Roos \cite{Ro} and Quillen \cite{Q}, MacLane cohomology of rings \cite{JP}, \cite{ML}, \cite{pw} and the cohomology of algebraic theories \cite{JP}.

In this article we analyze the functoriality of Baues-Wirsching cohomology with respect to functors between small categories. Namely, for a functor $u: \Ee\rightarrow \Bb$, we construct natural spectral sequences relating the Baues-Wirsching cohomology of the category $\Ee$ with that of the category $\Bb$ and associated local fiber data. We derive a first quadrant cohomology spectral sequence of the form
$$E_2^{p,q}\cong H_{BW}^p(\Bb, \h^q(-/Fu, D\circ Q^{(-)}))\Rightarrow H_{BW}^{p+q}(\Ee, D)$$
which is functorial with respect to natural transformations and where  
$$\h^q(-/Fu, D\circ Q^{(-)})= \llim^q_{-/Fu}(D\circ Q^{(-)}): F\Bb\rightarrow \Aa$$
is a particular natural system keeping track of the fiber data. We also introduce the dual concept of Baues-Wirsching homology of small categories and derive a homology spectral sequence dual to the cohomology spectral sequence above. Specializing to particular coefficient systems, for example given by bimodules, modules, local or trivial systems, then induces spectral sequences for Hochschild-Mitchell cohomology and homology and for the other classical cohomology and homology theories of small categories. An interpretation of the cohomology of algebraic theories in terms of Baues-Wirsching cohomology due to Jibladze and Pirashvili \cite{JP} allows us also to study morphisms of algebraic theories cohomologically.

The main ingredient in the construction of these spectral sequences is the Andr\'e spectral sequence \cite{An}. Given a functor between small categories, Andr\'e constructed a spectral sequence for the cohomology and homology of the categories involved. Using an interpretation of the Baues-Wirsching cohomology and homology groups as particular Ext and Tor groups in appropriate functor categories we extend the Andr\'e spectral sequences to the context of Baues-Wirsching theories and unify the existing spectral sequence constructions for cohomology and homology of categories studied before. Special cases of these spectral sequences were also constructed by Cegarra \cite{Ce} and recently by Robinson \cite{Rb} to study the functoriality of cohomology of diagrams of groups and algebras.

In the particular case that the functor $u: \Ee\rightarrow \Bb$ is a Grothendieck fibration of small categories, the local categorical data can be identified more explicitly in terms of fiber data. The cohomology spectral sequence for a Grothendieck fibration turns out to be equivalent to the one constructed by Pirashvili and Redondo \cite{pr}, which converges to the cohomology of the Grothendieck construction of a pseudofunctor. We again analyze several special cases, including a general Cartan-Leray type spectral sequence for categorical group actions. In the particular case of bimodule coefficients such a Cartan-Leray type spectral sequence was also studied by Ciblis and Redondo \cite{CR}, \cite{pr} for Galois coverings of $k$-linear categories, which proves to be an important calculational tool for the representation theory of finite-dimensional algebras.

In a sequel to this paper we aim to study cohomology and homology theories 
with even more general coefficient systems, due to Thomason \cite{Wei2}, 
given by functors from the comma category $\Delta/\mathscr{C}$ 
into abelian groups or more generally into a model category, 
where $\mathscr{C}$ is the given small category 
and $\Delta$ the simplex category.  
These will allow for the construction of spectral sequences 
generalizing the classical Bousfield-Kan
homotopy limit and colimit spectral sequences.

\section{Constructing spectral sequences for functors}
\subsection{Constructing Andr\'e spectral sequences} 

In  this section we review and unify the constructions of various spectral sequences for functors between small categories. 

For a given small category $\Cc$ and an abelian category $\Aa$, for example if $\Aa=\Ab$ is the category of abelian groups, we can do homological algebra in the functor category $\Fun(\Cc, \Aa)$. It turns out that $\Fun(\Cc, \Aa)$ is again an abelian category with enough projective objects and by a theorem of Grothendieck \cite{G} also with enough injective objects and so the classical functorial constructions of homological algebra like $\Ext$ and $\Tor$  can be performed in $\Fun(\Cc, \Aa)$ (see \cite{Fr}, \cite{G}). 

Let us first recall the definitions of the cohomology and homology of a small category via derived limits and colimits, and the construction of the Andr\'e spectral sequences for general functors between small categories (see \cite{An}, \cite{Hu}, and \cite{Ob}). 

\begin{definition} Let $\Cc$ be a small category and $F:\Cc\rightarrow \Aa$ a functor to an abelian category $\Aa$. 
\begin{itemize}
\item [(i)] Assume $\Aa$ is complete and has exact products. The {\it $n$-th  cohomology of $\Cc$ with coefficients in $F$} is defined as:
$$H^n(\Cc, F)=\llim^{n}_{\Cc} F.$$
\item [(ii)] Assume $\Aa$ is cocomplete and has exact coproducts. The {\it $n$-th homology of $\Cc$ with coefficients in $F$} is defined as:
$$H_n(\Cc, F)=\clim_{n}^{\Cc} F.$$
\end{itemize}
\end{definition}

Now let $u: \Ee\rightarrow \Bb$ be a functor between two small categories $\Ee$ and $\Bb$ and $F:\Bb\rightarrow \Aa$ a functor from $\Bb$ to any category  $\Aa$. We get an induced functor between functor categories
$$u^*: \Fun(\Bb, \Aa)\rightarrow \Fun(\Ee, \Aa), \,\, u^*(F)=F\circ u.$$
If the category $\Aa$ is complete then the functor $u^*$ has a right adjoint functor, the right Kan extension,
$$Ran_u: \Fun(\Ee, \Aa)\rightarrow \Fun(\Bb, \Aa).$$ 
If $\Aa$ is moreover an abelian category with exact products, then $Ran_u$ has right satellites (see \cite{CE}),
$$R^qu_*:=Ran_u^q: \Fun(\Ee, \Aa)\rightarrow \Fun(\Bb, \Aa).$$
In this general situation we have the following cohomology spectral sequence due to Andr\'e \cite{An} (see also \cite{Hu}, \cite{Ob}).

\begin{theorem}\label{H^*andre}
Let $u: \Ee\rightarrow \Bb$ be a functor between small categories and $F:\Ee\rightarrow \Aa$ a functor to a complete abelian category $\Aa$ with exact products.
Then there exists a first quadrant spectral sequence of the form
$$E_2^{p,q}\cong H^p(\Bb, (R^qu_*)(F))\Rightarrow H^{p+q}(\Ee, F)$$
which is functorial with respect to natural transformations, where $R^qu_*=Ran^q_{u}$ is the $q$-th right satellite of  $Ran_{u}$, the right Kan extension
along the functor $u$.
\end{theorem}

For every functor $u: \Ee\rightarrow \Bb$ and object $b\in \Bb$ let 
$Q^b: b/u\rightarrow \Ee$ be the forgetful functor from the comma category $b/u$ to $\Ee$. 
We can identify the $q$-th right satellite $R^qu_*=Ran^q_{u}$ in terms of the comma category as follows (this is the dual statement of \cite[Application 2]{GZ}):
$$(Ran^q_{u}F)(b)\cong \llim^q_{b/u}(F\circ Q^b).$$

These isomorphisms are natural in objects $b$ and functors $F$ of $\Fun(\Ee, \Aa)$. 
Therefore we can identify the $E_2$-term of the Andr\'e spectral sequence and obtain the following:

\begin{corollary}\label{H^*andre+GZ}
Let $u: \Ee\rightarrow \Bb$ be a functor between small categories and $F:\Ee\rightarrow \Aa$ a functor to a complete abelian category $\Aa$ with exact products. Then there exists a first quadrant spectral sequence of the form
$$E_2^{p,q}\cong H^p(\Bb, \h^q(-/u, F\circ Q^{(-)}))\Rightarrow H^{p+q}(\Ee, F)$$
which is functorial with respect to natural transformations and where  
$$\h^q(-/u, F\circ Q^{(-)})= \llim^q_{-/u}(F\circ Q^{(-)}): \Bb\rightarrow \Aa.$$
\end{corollary}

We can also consider the dual situation. If the abelian category $\Aa$ is cocomplete then the functor $u^*$ has a left adjoint functor, the left Kan extension,
$$Lan^u: \Fun(\Ee, \Aa)\rightarrow \Fun(\Bb, \Aa).$$ 
If $\Aa$ is moreover an abelian category with exact coproducts then $Lan^u$ has left satellites (see \cite{CE}),
$$L_qu_*:=Lan^u_q: \Fun(\Ee, \Aa)\rightarrow \Fun(\Bb, \Aa),$$
and we get the following homology Andr\'e spectral sequence:

\begin{theorem}\label{H_*andre}
Let $u: \Ee\rightarrow \Bb$ be a functor between small categories and $F:\Ee\rightarrow \Aa$  a functor to a cocomplete abelian category $\Aa$ with exact coproducts.
Then there exists a third quadrant spectral sequence of the form
$$E^2_{p,q}\cong H_p(\Bb, (L_qu_*)(F))\Rightarrow H_{p+q}(\Ee, F)$$
which is functorial with respect to natural transformations, where $L_q u_*=Lan_q^{u}$ is the $q$-th left satellite of $Lan^{u}$, the left Kan extension along the functor $u$.
\end{theorem}

For every functor $u: \Ee\rightarrow \Bb$ and every given object $b\in \Bb$ let 
$Q_b: u/b\rightarrow \Ee$ be now the forgetful functor from the comma category $u/b$ to $\Ee$. 
We can then identify the $q$-th left satellite $L_qu_*=Lan_q^{u}$ as follows (see \cite[Application 2]{GZ}):
$$(Lan_q^{u}F)(b)\cong \clim_q^{u/b}(F\circ Q_b).$$

These isomorphisms are again natural in objects $b$ and functors $F$ of $\Fun(\Ee, \Aa)$ and we can identify the $E_2$-term of the above spectral sequence and obtain:

\begin{corollary}\label{H_*andre+GZ}
Let $u: \Ee\rightarrow \Bb$ be a functor between small categories, $\Aa$ a cocomplete abelian category with exact coproducts and $F:\Ee\rightarrow \Aa$ a functor. Then there exists a third quadrant spectral sequence of the form
$$E_2^{p,q}\cong H_p(\Bb, \h_q(u/-, F\circ Q_{(-)}))\Rightarrow H_{p+q}(\Ee, F)$$
which is functorial with respect to natural transformations and where  
$$\h_q(u/-, F\circ Q_{(-)})= \clim_q^{u/-}(F\circ Q_{(-)}): \Bb\rightarrow \Aa.$$
\end{corollary}

\subsection{Constructing spectral sequences for Baues-Wirsching cohomology and homology}

Let us now recall the definition of the Baues-Wirsching cohomology of a small category and its main properties \cite{BW}. Dually, we will also introduce the notion of Baues-Wirsching homology.

\begin{definition} Let $\Cc$ be a small category. The {\it factorization category} $F\Cc$ of $\Cc$ is the category whose object set is the set of morphisms of $\Cc$ and whose Hom-sets $F\Cc(f, f')$ are the sets of pairs $(\alpha, \beta)$ such that $f'=\beta f\alpha$,
 $$
\xymatrix{b\rto^\beta &b'\\
a\uto^{f}& a'\uto_{f'}\lto^\alpha.}
$$
The composition of morphisms in the factorization category $F\Cc$ is defined by $$(\alpha', \beta')\circ (\alpha, \beta)= (\alpha\circ \alpha', \beta'\circ \beta).$$
\end{definition}

We will introduce the following general coefficient systems:

\begin{definition}\label{def:structure}
Let $\Mm$ be a category. A functor $D: F\Cc\rightarrow \Mm$ is called a {\it natural system with values in} $\Mm$. For $\alpha,f,\beta$ as above we write 
\begin{align}\label{structure}
\alpha^*&=D(\alpha,1):D(f)\to D(f\circ\alpha)\\
\beta_*&=D(1,\beta):D(f)\to D(\beta\circ f).
\end{align}
\end{definition}

The factorization category construction is functorial on the category $\Cat$ of small categories. The functor
$$F: \Cat\rightarrow \Cat$$
is given on morphisms as follows: if $u: \Ee\rightarrow \Bb$ is a functor then
$Fu:F\Ee\rightarrow F\Bb$ is the functor given on objects by $Fu(f)=u(f)$ and on morphisms by
$Fu(\alpha, \beta)=(u(\alpha), u(\beta))$. 

There is a functor $$F\Cc\to\Cc^{\mathrm{op}}\times\Cc$$ given by $(a\stackrel f\longrightarrow b)\mapsto(a,b)$ and $(\alpha,\beta)\mapsto(\alpha,\beta)$. This construction is natural in $\Cc$. 

We define now the Baues-Wirsching cohomology of a small category as follows:

\begin{definition}\label{bwcochn}
Let $\Cc$ be a small category and let $D: F\Cc\rightarrow \Aa$ be a natural system with values in a complete abelian category $\Aa$ with exact products. 
We define the {\it Baues-Wirsching cochain complex} $C_{BW}^*(\Cc, D)$ as follows: For each integer $n\geq 0$ the $n$-th cochain object is given by
$$C_{BW}^n(\Cc, D)= \prod_{C_0\stackrel{f_1}\leftarrow C_1\stackrel{f_2}\leftarrow\cdots \stackrel{f_n}\leftarrow C_n} D(f_1\circ f_2\circ \cdots\circ f_n).$$
Let $N_{\bullet}\Cc$ be the nerve of $\Cc$. Considering elements of $C_{BW}^n(\Cc, D)$ as maps $\sigma: N_n\Cc\rightarrow \bigcup_{f\in Mor(\Cc)}D(f)$ with 
$\sigma(f_1, \ldots, f_n)\in D(f_1\circ\cdots\circ f_n)$, where $\sigma(1_C)\in D(1_C)$ for $n=0$, the differential 
$$d : C_{BW}^n(\Cc,D) \to C_{BW}^{n+1}(\Cc,D)$$
is given by
$$\begin{array}{ll}
(d\sigma)(f_1, \cdots, f_{n+1}) &= (f_1)_* \sigma(f_2,
\cdots, f_{n+1}) \\ 
& + \sum_{i=1}^n (-1)^i \sigma( f_1, \cdots, f_i\circ
f_{i+1}, \cdots, f_{n+1}) \\
& + (-1)^{n+1} (f_{n+1})^* \sigma(f_1, \cdots, f_{n}).
\end{array}$$
The {\it $n$-th Baues-Wirsching cohomology} is defined as
$$H^n_{BW}(\Cc, D)=H^n(C_{BW}^*(\Cc, D), d).$$
\end{definition}

Dually, we can define the Baues-Wirsching homology of a small category. 

\begin{definition}
Let $\Cc$ be a small category and let $D: F\Cc\rightarrow \Aa$ be a natural system with values in a cocomplete abelian category $\Aa$ with exact coproducts. 
We define the {\it Baues-Wirsching chain complex} $C^{BW}_*(\Cc, D)$ as follows: For each integer $n\geq 0$ the $n$-th chain object is given by
$$C^{BW}_n(\Cc, D)= \bigoplus_{C_0\stackrel{f_1}\leftarrow C_1\stackrel{f_2}\leftarrow\cdots \stackrel{f_n}\leftarrow C_n} D(f_1\circ f_2\circ \cdots\circ f_n).$$

Let $\lambda=(C_0\stackrel{f_1}\leftarrow C_1\stackrel{f_2}\leftarrow\cdots \stackrel{f_n}\leftarrow C_n)$ be a string of $n$ composable arrows and let $in_{\lambda}: D(\lambda)\hookrightarrow C^{BW}_n(\Cc, D)$ be the inclusion. We first define
$$d_i: C^{BW}_n(\Cc, D)\rightarrow C^{BW}_{n-1}(\Cc, D), \,\,\, 0\leq i \leq n.$$
by setting
\[
d_i\circ in_{\lambda}=
\begin{cases} 
in_{d_0\lambda}\circ D(id_{C_0}\circ f_n), & \text{if } i=0,\\
in_{d_i\lambda}, & \text{if } 0<i<n,\\
in_{d_n\lambda}\circ D(f_0\circ id_{C_n}), & \text{if } i=n.
\end{cases}
\]
Here $d_i\lambda$ is the $i$-th face of the string in the nerve $N_{\bullet}\Cc$. The differential $$d:C^{BW}_n(\Cc, D)\rightarrow C^{BW}_{n-1}(\Cc, D)$$ is given by the alternating sum
$$d=\sum_{i=0}^n (-1)^i d_i.$$
The {\it $n$-th Baues-Wirsching homology} is defined as
$$H_n^{BW}(\Cc, D)=H_n(C^{BW}_*(\Cc, D), d).$$
\end{definition}

\begin{remark} 
It follows from the definition, that even though the coefficient systems are more general, Baues-Wirsching cohomology can also be interpreted in terms of classical cohomology of small categories over the factorization category $F\Cc$, namely we have an isomorphism
$$H^n_{BW}(\Cc, D)\cong H^n(F\Cc, D).$$
In fact, we can identify Baues-Wirsching cohomology as the derived functors of the limit functor~\cite[Theorem 4.4 and Remark 8.7]{BW}, 
$$H^n_{BW}(\Cc, D)\cong \Ext^n_{F\Cc}(\Z, D)\cong\llim_{F\Cc}^n D,$$
where $\Z: F\Cc\rightarrow \Aa$ is the constant natural system on $\Cc$. 

Dually, we can identify Baues-Wirsching homology as the derived functors of the colimit functor 
$$H_n^{BW}(\Cc, D)\cong \Tor_n^{F\Cc}(\Z, D)\cong\clim^{F\Cc}_n D$$
and therefore get again an isomorphism
$$H^{BW}_n(\Cc, D)\cong H_n(F\Cc, D).$$ 
The Baues-Wirsching cochain complex can also be seen naturally as a functor from the category $\Nat$ of natural systems given by all pairs $(\Cc, D)$, where $\Cc$ is a small category and $D$ a natural systems into the category of cochain complexes. Similarly Baues-Wirsching cohomology gives an induced functor from $\Nat$ into the category of graded abelian groups $\Ab$. 

In particular let us note that any equivalence of small categories $\phi: \Cc'\rightarrow \Cc$ induces an isomorphism of Baues-Wirsching cohomology groups (see \cite[Theorem 1.11]{BW}):
$$\phi^*: H^n_{BW}(\Cc, D)\stackrel{\cong}\rightarrow H^n_{BW}(\Cc', \phi^*D).$$

The Baues-Wirsching cochain complex can further be extended to give a $2$-functor from an extended $2$-category of natural systems into the $2$-category of cochain complexes such that the Baues-Wirsching cohomology factors through a quotient of $\Nat$ \cite{Mu}.

Dually, analogous functoriality properties can be obtained for the Baues-Wirsching chain complex and the associated Baues-Wirsching homology of a small category.
\end{remark}

Baues-Wirsching cohomology and homology also generalize many existing cohomology and homology theories whose coefficient systems can be seen as particular cases of natural systems. 

Let us recall the most important examples here (see also \cite{BW}). We have the following general diagram of categories and functors, in which $\Cc$ is a given small category:
$$F\Cc \stackrel{\pi}\longrightarrow\Cc^{op}\times \Cc\stackrel{p}\longrightarrow \Cc\stackrel{q}\longrightarrow\pi_1\Cc\stackrel{t}\longrightarrow \mathbb{1}$$
In this diagram, $\pi$ and $p$ are the forgetful functors and $q$ is the localization functor into the fundamental groupoid $\pi_1\Cc=(\Mor \Cc)^{-1}\Cc$ of $\Cc$ (see \cite{GZ}). Furthermore, $\mathbb{1}$ is the the trivial category consisting of one object and one morphism and $t$ is the trivial functor. 

Now let $\Aa$ be a
(co)complete abelian category with exact (co)products, as above. Pulling back functors from $\Fun(\Cc', \Aa)$ via the functors in the above diagram, where $\Cc'$ is any one of the categories in the diagram, induces natural systems on the category $\Cc$ and we can define various versions of cohomology and homology of small categories, which can all be seen as special cases of Baues-Wirsching cohomology or homology in this way (see \cite[Definition 1.18]{BW}).

\begin{definition} 
Let $\Aa$ be a complete abelian category with exact products for cohomology, or a cocomplete abelian category with exact coproducts for homology. Then:
\begin{itemize}
\item[(i)] $M$ is called a {\it $\Cc$-bimodule} if $M$ is a functor of $\Fun(\Cc^{op}\times \Cc, \Aa)$. Define the cohomology $H_{HM}^*(\Cc, M)=H^*_{BW}(\Cc, \pi^*M)$ and dually the homology $H^{HM}_*(\Cc, M)=H_*^{BW}(\Cc, \pi^*M)$.
\item[(ii)] $F$ is called a {\it $\Cc$-module} if $F$ is a functor of $\Fun(\Cc, \Aa)$. Define the cohomology $H^*(\Cc, F)=H^*_{BW}(\Cc, \pi^*p^*F)$ and dually the homology $H_*(\Cc, F)=H_*^{BW}(\Cc, \pi^*p^*F)$.
\item[(iii)] $L$ is called a {\it local system} on $\Cc$ if $L$ is a functor of $\Fun(\pi_1\Cc, \Aa)$. Define the cohomology $H^*(\Cc, L)=H^*_{BW}(\Cc, \pi^*p^*q^*L)$ and dually the homology $H_*(\Cc, L)=H_*^{BW}(\Cc, \pi^*p^*q^*L)$.
\item[(iv)] $A$ is a {\it trivial system} on $\Cc$ if $A$ is an abelian group (resp. an object in $\Aa$), i.~e. a functor of $\Fun(\mathbb{1}, \Ab)$ (resp. a functor of 
$\Fun(\mathbb{1}, \Aa)$). Define the cohomology $H^*(\Cc, A)=H^*_{BW}(\Cc, \pi^*p^*q^*t^*A)$ and dually the homology $H_*(\Cc, A)=H_*^{BW}(\Cc, \pi^*p^*q^*t^*A)$.
\end{itemize}
\end{definition}

These various cohomology and homology theories can be identified with the ones known previously in the literature. We have the following result (see \cite[Section 8]{BW}).

\begin{proposition} Let $\Aa$ be a complete abelian category with exact products. Then:
\begin{itemize}
\item[(i)] For any $\Cc$-bimodule $M$ of $\Fun(\Cc^{op}\times \Cc, \Aa)$ we have:\\
$H_{HM}^n(\Cc, M)\cong \Ext_{F\Cc}^n(\Z, \pi^*M)\cong\Ext^n_{\Cc^{op}\times \Cc}(\Z\Cc, M).$
\item[(ii)] For any $\Cc$-module $F$ of $\Fun(\Cc, \Aa)$ we have:\\
$H^n(\Cc, F)\cong \Ext_{F\Cc}^n(\Z, \pi^*p^*F)\cong\Ext^n_{\Cc}(\Z, F)\cong{\lim}^n_{\Cc} F.$
\item[(iii)] For any local system $L$ on $\Cc$ of $\Fun(\pi_1\Cc, \Aa)$ we have:\\
$H^n(\Cc, L)\cong \Ext_{F\Cc}^n(\Z, \pi^*p^*q^*L)\cong \Ext^n_{\pi_1\Cc}(\Z, L)\\\hspace*{1.7cm}\cong {\lim}^n_{\pi_1\Cc}L\cong H^n(B\Cc, L)$\\
where $B\Cc$ is the classifying space of the small category $\Cc$.
\end{itemize}
\end{proposition} 

\begin{proof}
The statements of (i) and (ii) follow immediately from \cite[Proposition 8.5]{BW}. The canonical $\Cc$-bimodule $\Z\Cc: \Cc^{op}\times \Cc\rightarrow \Aa$ used in (i) carries any object $(A, B)$ to the free abelian group generated by the morphism set $\Hom_{\Cc}(A, B)$. The statement of (iii) is due to Quillen \cite[\S 1, p. 83, (1)]{Q}.
\end{proof}

Dually, we have a similar characterization for the various homology theories of a small category. 

\begin{proposition} Let $\Aa$ be a cocomplete abelian category with exact coproducts. Then:
\begin{itemize}
\item[(i)] For any $\Cc$-bimodule $M$ of $\Fun(\Cc^{op}\times \Cc, \Aa)$ we have:\\
$H^{HM}_n(\Cc, M)\cong \Tor^{F\Cc}_n(\Z, \pi^*M)\cong\Tor_n^{\Cc^{op}\times \Cc}(\Z\Cc, M).$
\item[(ii)] For any $\Cc$-module $F$ of $\Fun(\Cc, \Aa)$ we have:\\
$H_n(\Cc, F)\cong \Tor^{F\Cc}_n(\Z, \pi^*p^*F)\cong\Tor_n^{\Cc}(\Z, F)\cong{\clim}_n^{\Cc} F.$
\item[(iii)] For any local system $L$ on $\Cc$ of $\Fun(\pi_1\Cc, \Aa)$ we have:\\
$H_n(\Cc, L)\cong \Tor^{F\Cc}_n(\Z, \pi^*p^*q^*L)\cong \Tor_n^{\pi_1\Cc}(\Z, L)\\\hspace*{1.7cm}\cong {\clim}_n^{\pi_1\Cc}L\cong H_n(B\Cc, L)$\\
where $B\Cc$ is the classifying space of the small category $\Cc$.
\end{itemize}
\end{proposition}

\begin{proof}
The proofs of (i) and (ii) are just dual to those of the preceding theorem and the statement of (iii) was proven by Quillen \cite[\S 1, p.83, (2)]{Q}.
\end{proof}

The cohomology and homology theories in (i) for $\Cc$-bimodule coefficients are the ones introduced by Hochschild and Mitchell, and which are also referred to as 
{\it Hochschild-Mitchell cohomology} \cite{CWMi,Mi}. Hochschild-Mitchell homology of categories was also introduced and studied before by Pirashvili and Waldhausen \cite{pw}. Special cases of (ii) have been studied by Watts \cite{Wa}, Roos \cite{Ro} and Quillen \cite{Q}. 
These in turn also naturally generalize group cohomology and homology, where a group $G$ is simply seen as a category with one object and morphism set $G$.

We will now derive several spectral sequences for Baues-Wirsching cohomology and homology from the respective Andr\'e spectral sequences as constructed in the first paragraph.

Let us consider a functor $u:\Ee\to \Bb$ between small categories and the following diagram of categories and functors as defined before:
$$
\xymatrix{F\Ee\rrto^{\pi}\dto^{Fu}&&\Ee^{op}\times \Ee\dto^{u^{op}\times u}\rrto^{p}&& \Ee\dto^u\rrto^{q}&&\pi_1\Ee\dto^{\pi_1(u)}\rrto^{t}&&\mathbb{1}\ar@{=}[d]\\
                  F\Bb\rrto^{\tilde{\pi}}               &&\Bb^{op}\times \Bb\rrto^{\tilde{p}}&&\Bb\rrto^{\tilde{q}}            &&\pi_1\Bb\rrto^{\tilde{t}}      &&\mathbb{1}}
$$

In order to deal with all the different cases of coefficient systems at once, let us write $\Cc$ for any of the categories in the lower row of the above diagram and denote by
$u':F\Ee\to \Cc$ the associated composition of functors. Furthermore, let $\Aa$ be a complete abelian category with exact products. 
We have in each case a diagram:
 $$
\xymatrix{
*+++{\Fun(F\Ee,\Aa)}\ar@<-5pt>[rrrr]_{u'_*} \ar@<5pt>[rrdd]^-{\lim_{F\Ee}}
&&&& *+++{\Fun(\Cc,\Aa)} \ar@<-5pt>[llll]_{{u'}^*}   \ar@<5pt>[lldd]^-{\lim_{\Cc}} \\ \\
&&\Aa\ar@<5pt>[rruu]^-c \ar@<5pt>[lluu]^-c
}
$$
where $c$ denotes the constant diagram functor, ${u'}^*$ is pre-composition with $u'$, and the other functors in the diagram are the right adjoints of these, given by the limits $\lim_{F\Ee}$, $\lim_{\Cc}$ and by $u'_*=Ran_{u'}$. 

The spectral sequence for the derived functors 
of the composite functor 
$$\llim_{F\Ee}(-)=\llim_\Cc u'_ *(-)$$
is an Andr\'e spectral sequence as considered above. It converges to the Baues-Wirsching cohomology of $\Ee$ with coefficients $D$ of $\Fun(F\Ee,\Aa)$ and therefore Theorem \ref{H^*andre} gives a first quadrant cohomology spectral sequence of the form:
$$
E_2^{p,q}\cong H^p(\Cc, Ran^q_{u'}(D))\Rightarrow H^{p+q}_{BW}(\Ee, D)
$$
Identifying the terms in the $E_2$-page according to Corollary \ref{H^*andre+GZ} we therefore get:
$$
E_2^{p,q}\cong H^p(\Cc,  \h^q(-/u', D\circ Q^{(-)}))\Rightarrow H^{p+q}_{BW}(\Ee, D)
$$
where $D\circ Q^{(-)}$ is the composition of functors
$$-/u'\stackrel{Q^{(-)}}\longrightarrow F\Ee\stackrel{D}\longrightarrow \Aa.$$

In particular cases the $E_2$-page may be simplified.
For example in the case $\Cc=F\Bb$ with $u'=Fu$ considered above, we get the
following spectral sequence:

\begin{theorem}\label{H^*BW} 
Let $\Ee$ and $\Bb$ be small categories and $u: \Ee\rightarrow \Bb$ a functor. Let $\Aa$ be a complete abelian category with exact products. Given a natural system $D: F\Ee\rightarrow \Aa$ on $\Ee$, there is a first quadrant cohomology spectral sequence 
$$E_2^{p,q}\cong H_{BW}^p(\Bb, (R^q Fu_*)(D))\Rightarrow H^{p+q}_{BW}(\Ee, D)$$
which is functorial with respect to natural transformations and where $R^qFu_*=Ran^q_{Fu}$ is the $q$-th right satellite of $Ran_{Fu}$, the right Kan extension 
along $Fu$.
\end{theorem}

Using the identification of the terms in the $E_2$-page as above we get therefore:

\begin{corollary}\label{H^*BW+GZ} 
Let $u: \Ee\rightarrow \Bb$ be a functor between small categories and $\Aa$ a complete abelian category with exact products. 
Let $D:F\Ee\rightarrow \Aa$ be a natural system on $\Ee$.
Then there exists a first quadrant cohomology spectral sequence of the form
$$E_2^{p,q}\cong H_{BW}^p(\Bb, \h^q(-/Fu, D\circ Q^{(-)}))\Rightarrow H_{BW}^{p+q}(\Ee, D)$$
which is functorial with respect to natural transformations and where  
$$\h^q(-/Fu, D\circ Q^{(-)})= \llim^q_{-/Fu}(D\circ Q^{(-)}): F\Bb\rightarrow \Aa.$$
\end{corollary}

For the case $\Cc=\Bb^{op}\times \Bb$ and $u'=\tilde{\pi}\circ Fu$ with a natural system $D: F\Ee\rightarrow \Aa$ on $\Ee$ we get in particular the following spectral sequence:
$$E_2^{p,q}\cong H^p(\Bb^{op}\times \Bb, R^q (\tilde{\pi}\circ Fu)_*(D))\Rightarrow H^{p+q}_{BW}(\Ee, D)$$ 
which after identifying the various terms involved gives a spectral sequence for Baues-Wirsching cohomology in terms of Hochschild-Mitchell cohomology:

\begin{theorem}\label{Thm+HM}
Let $\Ee$ and $\Bb$ be small categories and $u: \Ee\rightarrow \Bb$ a functor. Let $\Aa$ be a complete abelian category with exact products. Given a natural system $D: F\Ee\rightarrow \Aa$, there is a first quadrant cohomology spectral sequence 
$$E_2^{p,q}\cong H_{HM}^p(\Bb, R^q ((u^{op}\times u)\circ \pi)_*(D))\Rightarrow H^{p+q}_{BW}(\Ee, D)$$
which is functorial with respect to natural transformations and where $R^q((u^{op}\times u)\circ \pi)_*=Ran^q_{(u^{op}\times u)\circ \pi}$ is given as the $q$-th right satellite of $Ran_{(u^{op}\times u)\circ \pi}$, the right Kan extension along $(u^{op}\times u)\circ \pi$.
\end{theorem}

As before we can identify the terms in the $E_2$-page of the spectral sequence with more concrete data and get:

\begin{corollary}\label{Cor+HM}
Let $u: \Ee\rightarrow \Bb$ be a functor between small categories. Let $\Aa$ be a complete abelian category with exact products. Given a natural system $D: F\Ee\rightarrow \Aa$, there is a first quadrant cohomology spectral sequence 
$$E_2^{p,q}\cong H_{HM}^p(\Bb, \h^q(-/(u^{op}\times u)\circ \pi, D\circ Q^{(-)}))\Rightarrow H_{BW}^{p+q}(\Ee, D)$$
which is functorial with respect to natural transformations and where  
$$\h^q(-/(u^{op}\times u)\circ \pi, D\circ Q^{(-)})= \llim^q_{-/(u^{op}\times u)\circ \pi} (D\circ Q^{(-)}): \Bb^{op}\times \Bb\rightarrow \Aa.$$
\end{corollary}

In the special case that $D=\pi^*M$ for a bimodule $M: \Ee^{op}\times \Ee\rightarrow \Aa$ on $\Ee$ we can further identify the $E_2$-term of the spectral sequence in Theorem \ref{Thm+HM} to obtain a spectral sequence for Hochschild-Mitchell cohomology of the form:
$$E_2^{p,q}\cong H_{HM}^p(\Bb, R^q (u^{op}\times u)_*(M))\Rightarrow H^{p+q}_{HM}(\Ee, M) $$
Similarly, we can identify the $E_2$-page in terms of local fiber data as above in Corollary \ref{Cor+HM}. This spectral sequence can also be derived directly from the cohomological Andr\'e spectral sequence of Theorem \ref{H^*andre} using the interpretation of Hochschild-Mitchell cohomology as an appropriate cohomology of small categories.\\

Dually, using similar arguments as above we can derive homological versions of the spectral sequences for Baues-Wirsching homology of categories.

\begin{theorem}\label{H_*BW}  Let $\Ee$ and $\Bb$ be small categories and $u: \Ee\rightarrow \Bb$ a functor. Let $\Aa$ be a cocomplete abelian category with exact coproducts. Given a natural system $D: F\Ee\rightarrow \Aa$ on $\Ee$, there is a third quadrant homology spectral sequence 
$$E^2_{p,q}\cong H^{BW}_p(\Bb, (L_q Fu^*)(D))\Rightarrow H_{p+q}^{BW}(\Ee, D)$$
which is functorial with respect to natural transformations, where $L_q Fu^*=Lan_q^{Fu}$ is the $q$-th left satellite of $Lan^{Fu}$, the left Kan extension 
along $Fu$.
\end{theorem}

\begin{proof}
This is simply dual to the statement of Theorem \ref{H^*BW} and follows from the dual Andr\'e homology spectral sequence involving the higher derived functors of $\clim^{F\Cc}_n D$
in the description of Baues-Wirsching homology $H^{BW}_n(\Cc, D)$ (see \cite{An}, \cite[p. 2291]{Hu} and \cite[Appendix II.3]{GZ}).
\end{proof}

We can further identify the $E^2$-page of this spectral sequence and get:

\begin{corollary}\label{H_*BW+GZ}  Let $\Ee$ and $\Bb$ be small categories and $u: \Ee\rightarrow \Bb$ a functor. Let $\Aa$ be a cocomplete abelian category with exact coproducts. Given a natural system $D: F\Ee\rightarrow \Aa$ on $\Ee$, there is a third quadrant cohomology spectral sequence 
$$E^2_{p,q}\cong H^{BW}_p(\Bb, \h_q(Fu/-, D\circ Q_{(-)}))\Rightarrow H_{p+q}^{BW}(\Ee, D)$$
which is functorial with respect to natural transformations and where  
$$\h_q(Fu/-, D\circ Q_{(-)})= \clim_q^{Fu/-}(D\circ Q_{(-)}): F\Bb\rightarrow \Aa.$$
\end{corollary}

\begin{proof} We can identify the $E^2$-term in the above homology spectral sequence as follows (see \cite{An}, \cite{Hu})
$$Lan_q^{Fu}(D)\cong \clim^{Fu/\beta}_q D\circ Q_*$$
which then gives the desired homology spectral sequence.
\end{proof}

For the special case $\Cc=\Bb^{op}\times \Bb$ and $u'=\tilde{\pi}\circ Fu$ with a natural system $D: F\Ee\rightarrow \Aa$ on $\Ee$ we get a homology spectral sequence 

$$E^2_{p,q}\cong H_p(\Bb^{op}\times \Bb, L_q (\tilde{\pi}\circ Fu)_*(D))\Rightarrow H_{p+q}^{BW}(\Ee, D)$$ 
which after identifying the various terms gives the following spectral sequence for Hochschild-Mitchell homology dual to the one of Theorem \ref{Thm+HM}:

\begin{theorem} 
Let $\Ee$ and $\Bb$ be small categories and $u: \Ee\rightarrow \Bb$ a functor. Let $\Aa$ be a cocomplete abelian category with exact coproducts. Given a natural system $D: F\Ee \rightarrow \Aa$, there is a third quadrant homology spectral sequence 
$$E^2_{p,q}\cong H^{HM}_p(\Bb, L_q ((u^{op}\times u)\circ \pi_*)(D))\Rightarrow H_{p+q}^{BW}(\Ee, D)$$
which is functorial with respect to natural transformations and where $L_q((u^{op}\times u)\circ \pi)_*=Lan_q^{(u^{op}\times u)\circ \pi}$ is given as the $q$-th left satellite of $Lan^{(u^{op}\times u)\circ \pi}$, the left Kan extension 
 along $(u^{op}\times u)\circ \pi$.
\end{theorem}

Similarly, we can again identify the terms in the $E_2$-page of the spectral sequence and obtain:

\begin{corollary}
Let $u: \Ee\rightarrow \Bb$ be a functor between small categories and $\Aa$ a cocomplete abelian category with exact coproducts. Given a natural system $D: F\Ee \rightarrow \Aa$, there is a third quadrant homology spectral sequence 
$$E^2_{p,q}\cong H^{HM}_p(\Bb, \h_q((u^{op}\times u)\circ \pi/-, D\circ Q_{(-)}))\Rightarrow H^{BW}_{p+q}(\Ee, D)$$
which is functorial with respect to natural transformations and where  
$$\h_q((u^{op}\times u)\circ \pi/-, D\circ Q_{(-)})= \clim_q^{(u^{op}\times u)\circ \pi/-}(D\circ Q_{(-)}): \Bb^{op}\times \Bb\rightarrow \Aa.$$
\end{corollary}

As above in the special case $D=\pi^*M$ for a bimodule $M: \Ee^{op}\times \Ee\rightarrow \Aa$ on $\Ee$ we can further identify the $E^2$-term of the spectral sequence similarly as in Theorem \ref{Thm+HM} to obtain a spectral sequence for Hochschild-Mitchell homology of the form:
$$E^2_{p,q}\cong H^{HM}_p(\Bb,  L_q ((u^{op}\times u)_*)(M))\Rightarrow H_{p+q}^{HM}(\Ee, M)$$   

And again we can now identify the $E^2$-page in terms of local fiber data as above in Corollary \ref{Cor+HM}. This spectral sequence could also be derived directly from the homological Andr\'e spectral sequence of Theorem \ref{H_*andre} using the interpretation of Hochschild-Mitchell homology as an appropriate homology of small categories.\\

With similar arguments as above, we can also derive cohomology and homology spectral sequences for more special coefficient systems like $\Ee$-modules, local systems or trivial systems, which we will leave to the interested reader. For the case of $\Ee$-modules we only note that one just recovers the cohomology and homology Andr\'e spectral sequences of Theorems \ref{H^*andre} and \ref{H_*andre}. In the case of local systems, the spectral sequences can also be interpreted as a Leray type spectral sequence for the induced map $Bu: B\Ee\rightarrow B\Bb$ between the classifying spaces of the categories $\Ee$ and $\Bb$.

\subsection{A spectral sequence for the cohomology of algebraic theories.}
Let us now consider an application to the cohomology of algebraic theories. Jibladze and Pirashvili \cite{JP} constructed a general cohomology theory for algebraic theories, which gives a well-behaved cohomology theory for rings and algebras. They also indicated the importance of having more general coefficient systems than just modules. One approach towards this is to interpret cohomology of algebraic theories again as an appropriate Baues-Wirsching cohomology with natural systems as coefficients, see \cite[Sect. 6]{JP}. One needs to use these more general coefficient systems, for example, when classifying extensions of algebraic theories.\\

Let us first recall the definition of an algebraic theory:

\begin{definition} 
A {\it (finitary) algebraic theory} is a category whose objects are the natural numbers $\nn$, denoted $\zero, \one, \two, \dots \en, \ldots ,$  equipped with distinguished isomorphisms 
$$\phi_n\colon\en\stackrel\cong\longrightarrow \one^n$$
between each object $\en$ and the product of $n$ copies of the object $\one$. 
Morphisms between algebraic theories are simply functors which are identities on objects and preserve finite products, i.\ e.\ functors preserving the morphisms $\phi_n$ for all $n$. The category of algebraic theories will be denoted by $\theories$. 
\end{definition}

Now let $\aaa$ be an algebraic theory. There is an equivalence of categories \cite[Sect. 6]{JP}
$$\Ab(\theories/\aaa)\cong \Nat^{th}(\aaa^{op})$$
where $\Ab(\theories/\aaa)$ is the category of internal abelian group objects in the comma category $\theories/\aaa$ and $\Nat^{th}(\aaa^{op})$ is the category of coproduct-preserving natural systems on the opposite category $\aaa^{op}$, that is, natural systems $D$ such that for any diagram in $\aaa$ of the form
$$
\xymatrix@R3
pt{&&&&y_1\\x\ar[rr]^-f&& y_1\times y_2\times \ldots\times y_n\ar[rru]^-{p_1}\ar[rrd]_-{p_n}
&&\vdots\\
&&&&y_n,}$$
where $p_i$ are the canonical projections, the  homomorphism
$$D(f^{op})\rightarrow D(f^{op}\circ {p_1}^{op})\oplus\cdots \oplus D(f^{op}\circ {p_n}^{op})$$
induced by the structure maps $(p_i^{op})^*$ of Definition \ref{def:structure} \eqref{structure} is an isomorphism.

The cohomology 
with coefficients $M$ in $\Ab(\theories/\aaa)$ 
of  an algebraic theory $\aaa$ 
is defined as the Baues-Wirsching cohomology
$$H^*(\aaa, M)=H^*_{BW}(\aaa^{op}, \tilde{M})$$
where $\tilde{M}$ is the natural system associated to $M$ by the above equivalence between the categories $\Ab(\theories/\aaa)$ and $\Nat^{th}(\aaa^{op})$.
The cohomology may be calculated with the complex given in Definition \ref{bwcochn} or using a smaller normalized complex as explained in~\cite{BT}.
Dually, we can also define the homology of algebraic theories as
$$H_*(\aaa, M)=H_*^{BW}(\aaa^{op}, \tilde{M}).$$

Now let $u: \eee\rightarrow \bbb$ be a morphism of algebraic theories and $M$ a coefficient system from $\Ab(\theories/\eee)$, then it follows from 
Corollary \ref{H^*BW+GZ} that there exists a first quadrant cohomology spectral sequence
$$E_2^{p, q}\cong H^p_{BW}(\bbb^{op}, \h^q(-/Fu^{op}, \tilde{M}\circ Q^{(-)}))\Rightarrow H^{p+q}(\eee, M).$$
Dually, from Corollary \ref{H_*BW+GZ} we get a third quadrant homology spectral sequence
$$E^2_{p, q}\cong H_p^{BW}(\bbb^{op}, \h_q(Fu^{op}/-, \tilde{M}\circ Q_{(-)}))\Rightarrow H_{p+q}(\eee, M).$$
In general though, the $E_2$-pages of these spectral sequences are only given as general Baues-Wirsching cohomology groups and might not always be identified with a certain cohomology or homology of the algebraic theory $\bbb$. 

\section{Constructing spectral sequences for Grothendieck fibrations}

\subsection{Grothendieck fibrations and spectral sequences}

In this section we will analyze the spectral sequence constructions in the particular situation of a Grothendieck fibration of small categories. Applying the spectral sequence constructions of the preceding sections to a functor which is a Grothendieck fibration allows to identify the $E_2$-pages with simpler cohomology or homology groups keeping track of the fiber data. 

Given any functor $u: \Ee\rightarrow \Bb$  and an object $b$ of $\Bb$, recall that the {\it fiber category}
$\Ee_b=u^{-1}(b)$ is the subcategory of $\Ee$ that fits into the pullback diagram
$$
\xymatrix{\Ee_b\rto
\dto&\Ee\dto^u\\ {*}\rto^b&\Bb.}
$$
The objects of $\Ee_b$ are those objects of  $\Ee$ which map onto $b$ via the functor $u$ and the morphisms are given by those which map to the identity $1_b$.

We now have the following notion of a fibration between small categories due to Grothendieck \cite[Expos\'e VI]{SGA}:
\begin{definition}
Let $\Ee$ and $\Bb$ be small categories. A {\it Grothendieck fibration} is a functor $u: \Ee\rightarrow \Bb$ such that the fibers
$\Ee_b=u^{-1}(b)$ depend contravariantly and pseudofunctorially on the objects $b$ of the category $\Bb$. The category $\Ee$ is also called a {\it category fibered over} $\Bb$.
\end{definition}

There are many equivalent explicit definitions of a Grothendieck fibration.
We only recall here from~\cite{Gr1}, that  $u: \Ee\rightarrow \Bb$ is a Grothendieck fibration if for each object $b$ of $\Bb$ the inclusion functor from the fiber into the comma category
\[
j_b: \Ee_b\to b/u, \quad e\mapsto (e, b\stackrel=\longrightarrow ue)
\]
is coreflexive, i.~e. has a right adjoint left inverse. We have the following general characterization of Grothendieck fibrations:

\begin{theorem}
There is an equivalence of $2$-categories
$$\Fib(\Bb)\stackrel{\simeq}\leftrightarrow \PsdFun(\Bb^{op}, \Cat)$$
between the $2$-category of Grothendieck fibrations $\Fib(\Bb)$ over a category $\Bb$ and the $2$-category of contravariant pseudofunctors $\PsdFun(\Bb^{op}, \Cat)$ from $\Bb$ 
to the category $\Cat$ of small categories.
\end{theorem}

\begin{proof} For a detailed proof we refer to \cite{SGA} or \cite[Vol.\ 2, 8.3.1]{Bo}.
The equivalence is induced by the Grothendieck construction
$$\int: \PsdFun(\Bb^{op}, \Cat)\rightarrow \Fib(\Bb), \,\,\,\, F\mapsto \int_{\Bb} F.$$
And if conversely $u: \Ee\rightarrow \Bb$ is a Grothendieck fibration the assignment 
$$G: \Bb^{op} \rightarrow \Cat, \, b\mapsto \Ee_b=u^{-1}(b).$$ 
defines a pseudofunctor.
\end{proof}

Let $u: \Ee\rightarrow \Bb$ be a Grothendieck fibration and $D: F\Ee \rightarrow \Aa$ be a natural system on the category $\Ee$, where $\Aa$ is a complete abelian category with exact products.

We get a local system 
$\h^q_{BW}(G(-), D|_{F\Ee_{(-)}}): \Bb\rightarrow \Aa$ from the associated pseudo\-functor $G: \Bb^{op} \rightarrow \Cat$ by assigning to every object $b$ of the category $\Bb$ the $q$-th Baues-Wirsching cohomology of the category $G(b)$
$$\h^q_{BW}(G(-), D|_{F\Ee_{(-)}}): \Bb\rightarrow \Aa, \,\,\, b\mapsto H_{BW}^q(G(b), D|_{F\Ee_b})$$
with coefficients in the natural system
$D|_{F\Ee_b}: F\Ee_b\rightarrow \Aa$.

For each object $b$ of the base category $\Bb$ 
we have a cartesian diagram
$$
\xymatrix{\Ee_b\dto_{=}\rto^
{j_b}&b/u\dto^{Q^b}\\
\Ee_b\dto\rto^
{i_b}\dto&\Ee\dto^u\\ {*}\rto^b&\Bb.}
$$
Let $R_{b}$ denote the right adjoint functor of the inclusion functor $j_b$
and let $D_b$ denote the natural system on $b/u$ given by the following composition,
$$D_b: F (b/u) \xrightarrow{F R_b} F\Ee_b \xrightarrow{F i_b} F\Ee \xrightarrow {D} \Aa.$$

Now we get a first quadrant cohomology spectral sequence from the cohomology spectral sequence for the Grothendieck construction of Pirashvili-Redondo \cite{pr} and the equivalence between Grothendieck fibrations and pseudofunctors:
$$E_2^{p,q}\cong H^p(\Bb, \h^q_{BW}(-/u, D_{(-)}))\Rightarrow H^{p+q}_{BW}(\Ee, D)$$

This spectral sequence is functorial with respect to $1$-morphisms, i.e. natural transformations between Grothendieck fibrations. 

To summarize, we have constructed the following cohomology spectral sequence for general Grothendieck fibrations of small categories:

\begin{theorem}\label{H^*BW+G} 
Let $\Ee$ and $\Bb$ be small categories and let $u: \Ee\rightarrow \Bb$ be a Grothendieck fibration. Given a natural system $D: F\Ee\rightarrow \Aa$ on $\Ee$, where $\Aa$ is a complete abelian category with exact products, there is a first quadrant spectral sequence 
$$E_2^{p,q}\cong H^p(\Bb, \h^q_{BW}(-/u, D_{(-)}))\Rightarrow H^{p+q}_{BW}(\Ee, D)$$
which is functorial with respect to $1$-morphisms of Grothendieck fibrations. 
\end{theorem}

In order to identify the $E_2$-term of this spectral sequence with local data of the fiber category we need the following definition, which corresponds to the property of h-locality introduced in \cite{pr}.

\begin{definition}
A natural system $D: F\Ee\rightarrow \Aa$ is called {\it local} if the adjoint functor $R_{b}$ of the inclusion functor $j_b: \Ee_b\to b/u$ induces an isomorphism in Baues-Wirsching cohomology
$$H^q_{BW}(b/u, D_b)\cong H^q_{BW}(\Ee_b, D\circ Fi_b)$$
for every $q$ and every object $b$ of the base category $\Bb$, i.~e. we have a natural isomorphism of local coefficient systems
$$\h^q_{BW}(-/u,  D_{(-)})\cong \h^q_{BW}(\Ee_{(-)}, D\circ Fi_{(-)}).$$
\end{definition}

Identifying the $E_2$-page of the above spectral sequence, we now have the following cohomology spectral sequence:

\begin{theorem}\label{H^*BW+Glocal} 
Let $\Ee$ and $\Bb$ be small categories and $u: \Ee\rightarrow \Bb$ be a Grothendieck fibration. Given a local natural system $D: F\Ee\rightarrow \Aa$ on $\Ee$, where $\Aa$ is a complete abelian category with exact products, there is a first quadrant spectral sequence 
$$E_2^{p,q}\cong H^p(\Bb, \h^q_{BW}(\Ee_{(-)}, D\circ Fi_{(-)}))\Rightarrow H^{p+q}_{BW}(\Ee, D)$$
with the local coefficient system 
$$\h^q_{BW}(\Ee_{(-)}, D\circ Fi_{(-)}): \Bb\rightarrow \Aa, \,\, b\mapsto H^q_{BW}(\Ee_b, D\circ Fi_b).$$
Furthermore, the spectral sequence is functorial with respect to $1$-morphisms of Grothendieck fibrations.
\end{theorem}

\begin{proof} 
This follows immediately from the construction of the cohomology spectral sequence in Theorem \ref{H^*BW+G} and the definition of a local natural system $D: F\Ee\rightarrow \Aa$ on the total category $\Ee$.
\end{proof}

Dually, we can also derive a homology version of the spectral sequence for a Grothendieck fibration by invoking the obvious dual notions and constructions involved:

\begin{theorem}\label{H_*BW+Glocal} Let $\Ee$ and $\Bb$ be small categories and $u: \Ee\rightarrow \Bb$ be a Grothendieck fibration. Given a colocal natural system $D: F\Ee\rightarrow \Aa$ on $\Ee$, where $\Aa$ is a cocomplete abelian category with exact coproducts, there is a third quadrant spectral sequence 
$$E^2_{p,q}\cong H_p(\Bb, \h_q^{BW}(\Ee_{(-)}, D\circ Fi_{(-)}))\Rightarrow H_{p+q}^{BW}(\Ee, D)$$
with the local coefficient system 
$$\h_q^{BW}(\Ee_{(-)}, D\circ Fi_{(-)}): \Bb\rightarrow \Aa, \,\, b\mapsto \h_q^{BW}(\Ee_b, D\circ Fi_b).$$
Furthermore, the spectral sequence is functorial with respect to $1$-morphisms of Grothendieck fibrations.
\end{theorem}

As in the preceding sections we can specialize these cohomology and homology spectral sequences for coefficient systems of bimodules, modules, local systems or trivial systems. The associated spectral sequences will then for example in the case of bimodules relate the Hochschild-Mitchell cohomology or homology groups in a Grothendieck fibration. We leave it to the interested reader to work out the details. 

\subsection{Applications and examples.} In this final section we will give interpretations and applications of particular spectral sequences, which were studied before in the literature. 

Formulated here in the language of Grothendieck fibrations, the cohomology spectral sequence of Theorem \ref{H^*BW+G} can also be seen as an equivalent version of the Pirashvili-Redondo spectral sequence, which converges to the Baues-Wirsching cohomology of the Grothendieck construction of a pseudofunctor  (see \cite{pr}, \cite{pr2}). And this spectral sequence again generalizes several others. Let us study some important examples.

\begin{example}
Let $\Cc$ be a small category with an action of a group $G$. The group $G$ can be considered as a small category with a single object $\star$ and the group $G$ as morphism set.
A $G$-action on $\Cc$ is then simply given as a functor
$$F: G^{op}\rightarrow \Cat,\,\,\, F(\star)=\Cc.$$ 
We can view $F$ as a pseudofunctor and the Grothendieck construction gives a small category $\int_G F.$
The above equivalence between pseudofunctors and Grothendieck fibrations therefore gives a Grothendieck fibration of small categories
$$u: \int_G F \rightarrow G.$$
Given in addition a natural system $D: \int_G F\rightarrow \Ab$, we get therefore a first quadrant cohomology spectral sequences of the form:
$$E_2^{p,q}\cong H^p(G, \h^q_{BW}(-/u, D_{(-)}))\Rightarrow H^{p+q}_{BW}(\int_G F, D).$$
Finally, we can identify the fiber data in this particular situation, which allows to rewrite the $E_2$-page and get:
$$E_2^{p,q}\cong H^p(G, H^q_{BW}(\Cc, D))\Rightarrow H^{p+q}_{BW}(\int_G F, D).$$
This can be interpreted as a categorical version of the classical Cartan-Leray spectral sequence, which was also derived as a special case in \cite[Remark 5.4]{pr} and for local coefficient systems in \cite[Proposition 5.4.2]{dH}.

Let us also study the particular situation of a $k$-linear category with a $G$-action, which is of importance in representation theory. 
Let $k$ be a field and $\Cc$ be a $k$-linear category, i.~e. a small category where morphism sets are $k$-vector spaces and composition is $k$-bilinear. 

Assume now that a group $G$ is acting freely on $\Cc$. As before we can take the Grothendieck construction which in this situation gives the quotient or orbit category $\int_G F= \Cc/G$. So we get as a special case a Cartan-Leray spectral sequence, which calculates the Baues-Wirsching cohomology of the quotient category:
$$E_2^{p,q}\cong H^p(G, H^q_{BW}(\Cc, D))\Rightarrow H^{p+q}_{BW}(\Cc/G, D).$$
For bimodule coefficients $M$ of $\Fun(\Cc/G, \Ab)$, this spectral sequence is the Cartan-Leray spectral sequence of Ciblis-Redondo \cite[Theorem 3.11]{CR} (see also \cite[Remark 5.3]{pr}), which calculates the Hochschild-Mitchell cohomology of the quotient category $\Cc/G$.

We can also derive dual homology versions of the above cohomology spectral sequences, but will leave these straightforward constructions to the interested reader.
\end{example}

Let us finally mention the case of cohomology and homology with local or trivial coefficient systems, which correspond to singular cohomology and homology of the associated classifying spaces with local or constant coefficients. 

\begin{example}
Let $u: \Ee\rightarrow \Bb$ be a Grothendieck fibration and $L: \pi_1\Ee\rightarrow \Ab$ a local system on $\Ee$. Then we get the following first quadrant cohomology spectral sequence from Theorem \ref{H^*BW+Glocal}:
$$E_2^{p,q}\cong H^p(B\Bb, \h^q(B\Ee_{(-)}, L_{(-)}))\Rightarrow H^{p+q}(B\Ee, L)$$
with the local system $\h^q(B\Ee_{(-)}, L_{(-)}): \pi_1\Bb\rightarrow \Ab$ as coefficients.

Dually, we have a third quadrant homology spectral sequence from Theorem \ref{H_*BW+Glocal}:
$$E^2_{p,q}\cong H_p(B\Bb, \h_q(B\Ee_{(-)}, L_{(-)}))\Rightarrow H_{p+q}(B\Ee, L)$$
with the local system $\h^q(B\Ee_{(-)}, L_{(-)}): \pi_1\Bb\rightarrow \Ab$ as coefficients.

Specializing even further, using trivial systems of coefficients in a constant abelian group $A$, i.~e. functors $A: \mathbb{1}\rightarrow \Ab$, we get immediately the following first quadrant cohomology spectral sequence:
$$E_2^{p,q}\cong H^p(B\Bb, \h^q(B\Ee_{(-)}, A))\Rightarrow H^{p+q}(B\Ee, A)$$

And dually again, we have a third quadrant homology spectral sequence:
$$E^2_{p,q}\cong H_p(B\Bb, \h_q(B\Ee_{(-)}, A))\Rightarrow H_{p+q}(B\Ee, A)$$
\end{example}

In general, starting with a Grothendieck fibration $u:\Ee\rightarrow \Bb$ of small categories and applying the classifying space functor $B(-)$ does not give a fibration or quasi-fibration of topological spaces. But del Hoyo \cite[Theorem 5.3.1]{dH} showed that if $u: \Ee\rightarrow \Bb$ is moreover a Quillen fibration of small categories, which means essentially that all the classifying spaces of the fibers $\Ee_b$ are weakly homotopy equivalent, then applying a {\it fibered} version $B_f(-)$ of the classifying space functor does indeed give a quasi-fibration of topological spaces and therefore the particular spectral sequences above are Serre type spectral sequences.

\vspace*{0.3cm} {\it Acknowledgements.} 
The first author was partially supported by the grants 
MTM2010-15831, 
MTM2010-20692, 
and SGR-1092-2009, 
and the third author by 
MTM2010-15831 
and SGR-119-2009. 
The second author likes to thank the Centre de Recerca Matem\`atica (CRM) in Bellaterra, Spain for inviting him during the research programme Homotopy Theory and Higher Categories (HOCAT).

\bigskip

\end{document}